\documentclass[11pt]{amsart}

\usepackage{amsmath}
\usepackage{amssymb}
\usepackage{epsfig}
\usepackage{amsfonts}

\newcommand{\excise}[1]{}

\newtheorem{thm}{Theorem}[section]
\newtheorem{lemma}[thm]{Lemma}

\newtheorem{cor}[thm]{Corollary}
\newtheorem{prop}[thm]{Proposition}
\newtheorem{conj}[thm]{Conjecture}
\newtheorem{conv}[thm]{Convention}
\newtheorem{op}[thm]{Open Problem}

\newtheorem{ex}[thm]{Example}
\newtheorem{rem}[thm]{Remark}
\newtheorem{Warn}[thm]{Caution}

    {\end{Example}}
    {\end{Remark}}

\def\wh{\widehat}

\def\emp{\nothing}

\def\la{\lambda}
\def\ga{\gamma}
\def\si{\sigma}

\def\al{\alpha}

\def\om{\omega}

\def\ve{\varepsilon}

\def\ssu{\subset}

\def\<{\langle}
\def\>{\rangle}

\def\Ups{\Phi}

\def\0{{\mathbf 0}}

\def\nothing{\varnothing}

\def\.{\hskip.06cm}
\def\ts{\hskip.03cm}

\def\ba{\textbf{\textit{a}}}

\def\RH{{\rm RH}}

\def\ds{{\text{\rm \ts d}}}
\def\sign{{\rm sign}}
\def\pd{{\pi'\hskip-.06cm}}
\def\vv{{\rm c}}
\def\mon{{\ga}}

\def\irredsym#1{\chi^{#1}}
\def\irredalt#1{\alpha^{#1}}
\def\epsi#1{\varepsilon_{#1}}

\def\res#1#2#3{{\sf Res}^{#1}_{#2}({#3})}

\def\la{\lambda}

\def\im{{\text{\rm Im}}}

\begin{document}
\title[Tensor square conjectures]{Kronecker products, characters,
partitions, \\ and the tensor square conjectures}

\author[Igor~Pak]{ \ Igor~Pak$^\star$}

\author[Greta~Panova]{ \ Greta~Panova$^\star$}

\author[Ernesto~Vallejo]{ \ Ernesto Vallejo$^\dagger$}

\date{\today}

\thanks{\thinspace ${\hspace{-.45ex}}^\star$Department of Mathematics, UCLA, Los Angeles, CA 90095, \ts
\texttt{\{pak,panova\}@math.ucla.edu}}

\thanks{\thinspace ${\hspace{-.45ex}}^\dagger$Universidad Nacional Aut\'onoma de M\'exico,
Centro de Ciencias Matem\'aticas, Apartado Postal 61-3 \\ {\rm  \  \hskip.5cm Xangari}, 58089 Morelia, Mich., Mexico, \ts
\texttt{vallejo@matmor.unam.mx}}

\begin{abstract}
We study the remarkable \emph{Saxl conjecture} which states that tensor squares of
certain irreducible representations of the symmetric groups~$S_n$ contain all
irreducibles as their constituents.  Our main result is that they contain
representations corresponding to hooks and two row Young diagrams.
For that, we develop a new sufficient
condition for the \emph{positivity of Kronecker coefficients} in terms of characters,
and use combinatorics of rim hook tableaux combined with known results
on unimodality of certain partition functions.  We also present connections
and speculations on random characters of~$S_n$.
\end{abstract}

\maketitle

\section{Introduction and main results}\label{s:intro}

\noindent
Different fields have different goals and different open problems.
Most of the time, fields peacefully coexist enriching each other
and the rest of mathematics.  But occasionally, a conjecture from
one field arises to present a difficult challenge in another, thus
exposing its technical strengths and weaknesses.  The story of
this paper is our effort in the face of one such challenge.

\smallskip

Motivated by \emph{John Thompson's conjecture} and \emph{Passman's problem}
(see~$\S$\ref{ss:fin-back}),
Heide, Saxl, Tiep and Zalesski recently proved that with a
few known exceptions, every irreducible character of a simple
group of Lie type is a constituent of the tensor square of the
Steinberg character~\cite{HSTZ}.  They conjecture that for every $n\ge 5$,
there is an irreducible character $\chi$ of~$A_n$ whose tensor
square $\chi \otimes \chi$ contains every irreducible character as a
constituent.\footnote{Authors of~\cite{HSTZ} report that this
conjecture was checked by Eamonn O'Brien for $n \le 17$.}
Here is the symmetric group analogue of this conjecture:

\begin{conj}[Tensor square conjecture]
For every $n\ge 3$, $n\neq 4, 9$, there is a partition $\mu \vdash n$, such that
tensor square of the irreducible character \. $\chi^\mu$ of~$S_n$
contains every irreducible character as a constituent.
\end{conj}

The \emph{Kronecker product problem} is a problem of computing
multiplicities
$$g(\la,\mu,\nu)= \langle \chi^\la, \chi^\mu\otimes\chi^\nu\rangle$$
of an irreducible character of $S_n$ in the tensor product of two others.
It is often referred as ``classic'', and ``one of the last major open problems''
in algebraic combinatorics~\cite{BWZ,Reg}.
Part of the problem is its imprecise statement: we are talking about
finding an explicit combinatorial interpretation here rather than
computational complexity (see Subsection~\ref{ss:fin-kron}).

Despite a large body of work on the Kronecker coefficients, both
classical and very recent (see e.g.~\cite{BO,Bla,BOR,Ike,Reg,Rem,RW,Val1,Val2}
and references therein), it is universally agreed that ``frustratingly little
is known about them''~\cite{Bur}.  Unfortunately, most results are limited
to partitions of very specific shape (hooks, two rows, etc.), and
the available tools are much too weak to resolve the tensor square
conjecture.

\smallskip

During a talk at UCLA, Jan Saxl made the following conjecture,
somewhat refining the tensor square conjecture.\footnote{UCLA
Combinatorics Seminar, Los Angeles, March 20, 2012.}

\begin{conj}[Saxl conjecture] \label{conj:saxl}
Denote by $\rho_k=(k,k-1,\ldots,2,1)\vdash n$, where $n=\binom{k+1}{2}$.
Then for every $k\ge 1$, the tensor square \. $\chi^{\rho_k}\otimes\chi^{\rho_k}$ \ts
contains every irreducible character of~$S_n$ as a constituent.
\end{conj}

Andrew Soffer checked the validity of conjecture for $k\le 8$.\footnote{Personal
communication.} While we believe the conjecture, we also realize that it
is beyond the reach of current technology.  In Section~\ref{s:known},
we briefly survey the implications of known tools towards the tensor
product and Saxl conjectures.  More importantly, we then develop a new tool,
aimed specifically at the Saxl conjecture:

\begin{lemma}[Main Lemma]\label{t:main}
Let $\mu =\mu'$ be a self-conjugate partition of~$n$, and let
$\nu=(2\ts\mu_1-1, 2\ts\mu_2-3, 2\ts\mu_3-5, \ldots)\vdash n$ be the partition
whose parts are lengths of the principal hooks of $\mu$.
Suppose $\chi^\la[\nu] \ne 0$ for some $\la\vdash n$.
Then $\chi^\la$ is a constituent of \. $\chi^{\mu}\otimes\chi^{\mu}$.
\end{lemma}

Curiously, the proof uses representation theory of~$A_n$ and is based on
the idea of the proof of~\cite[Thm~3.1]{BB} (see Section~\ref{s:main-proof}).
We use this theorem to obtain the following technical results (among others).

\begin{thm}\label{t:summary}
There is a universal constant~$L$, such that for every $k\ge L$, the tensor square
\. $\chi^{\rho_k}\otimes\chi^{\rho_k}$ \ts
contains characters $\chi^\la$ as constituents, for all
$$
\la \ts =\ts (n-\ell,\ell)\ts, \ \ \, 0\le \ell \le n/2\ts, \quad \ \text{or} \quad  \
\la \ts =\ts (n-r,1^r)\ts, \ \ \, 0\le r \le n-1.
$$
\end{thm}

Of course, this is only a first step towards proving the Saxl conjecture.  While the Main Lemma
is a powerful tool, proving that the characters are nonzero is also rather difficult in general,
due to the alternating signs in the Murnaghan--Nakayama rule.  We use a few known combinatorial
interpretations (for small values of~$\ell$), and rather technical known results
on monotonicity of the number of certain partitions (for larger~$\ell$),
to obtain the above theorem and a constellation of related results.

\smallskip

The rest of the paper is structured as follows.  We begin with a discussion of combinatorics
and asymptotics of the number of integer partitions in Section~\ref{s:part}.  We then turn
to characters of~$S_n$ and basic formulas for their computation in Section~\ref{s:char}.
There, we introduce two more shape sequences (\emph{chopped square}
and \emph{caret}), which will appear throughout the paper.  In the next
section (Section~\ref{s:known}), we present several known results on the Kronecker product,
and find easy applications to our problem.  In the following three sections we present
a large number of increasingly technical calculations evaluating the characters in terms
of certain partition functions, and using Main Lemma and known partition
inequalities to derive the results above.  In a short Section~\ref{s:rand}, we discuss
and largely speculate what happens for random characters.  We prove the Main Lemma in
Section~\ref{s:main-proof}, and conclude with final remarks.

\bigskip

\section{Integer partitions}\label{s:part}

\subsection{Basic definitions} \.
Let $\la\vdash n$ be a partition of~$n$, and let $P_n$ denote the set of
partitions of~$n$.  Denote by~$\la'$ the conjugate partition of $\la$.
Partition $\la$ is called \emph{self-conjugate} if $\la=\la'$.
Denote by $\ell(\la)=\la_1'$ the number of parts in~$\la$.

\medskip  \subsection{Asymptotics}\label{ss:part-as}
\. Let $\pi(n)=|P_n|$ be the number of partitions of~$n$.
Then
$$1+\sum_{n=1}^\infty \pi(n)\ts t^n \. = \. \prod_{i=1}^\infty \. \frac{1}{1-t^i}\.,
$$
and
$$\pi(n) \. \sim \. \frac{e^{c\sqrt{n}}}{4\sqrt{3}\ts n}\., \qquad \text{where} \ \ c\. = \.\pi \sqrt{\frac{2}{3}}\..
$$
Denote by $\pi_k(n)$ the number of partitions $\la \vdash kn$, such that their $k$-core is empty~\cite{Mac},
i.e.~there exists a rim hook tableau of shape~$\la$ and weight $(k^n)$.  Note that $\pi_1(n)=\pi(n)$.  Then
by~\cite[Exc.7.59e]{Sta}
$$
1+\sum_{n=1}^\infty \pi_k(n)\ts t^n \. = \. \prod_{i=1}^\infty \. \frac{1}{(1-t^i)^k}\.,
$$
and Lemma~4 in~\cite{LP} gives
$$\pi_k(n) \. \sim \. \left[\frac{k^{k+1}}{2^{3k+5}\ts 3^{k+1}}\right]^{1/4}\.
\frac{e^{c\sqrt{kn}}}{n^{(k+3)/4}}\., \quad \text{where} \ \ c\. = \.\pi \sqrt{\frac{2}{3}}\..
$$

\medskip
\subsection{Limit shapes}\label{ss:part-limit} \.
Let $\la\vdash n$ be a \emph{random partition on~$n$}, i.e. chosen uniformly at random from~$P_n$.
Scale by $1/\sqrt{n}$ the Young diagram $[\la]$ of a random partition.  It is known that for every
$\ve>0$, the scaled random shape is w.h.p. within $\ve$-distance from the curve
$$
e^{-cx/2}\. +\.
e^{-cy/2}\. =\. 1\ts,
$$
where~$c$ is as above~\cite{DVZ,Ver}. Somewhat loosely, we call such $\la$ the \emph{limit shape}.
Note that the limit shape is symmetric and has two infinite tails, so the longest
part and the number of parts we have $\la_1,\ell(\la)=\om(\sqrt{n})$.  In fact,
it is known that for random~$\la$, we have
$\la_1, \ell(\la) = c^{-1}\sqrt{n}\bigl(\log n+O(1)\bigr)$ w.h.p., as $n\to \infty$
(see~\cite{Fri}).

\medskip  \subsection{Partitions into infinite arithmetic progressions}\label{ss:basic-pi} \.
Fix $a,m \ge 1$, such that \ts gcd$(a,m)=1$.  Define integers $\pd_{a,m}(n)$ by
$$
\sum_{n=0}^\infty \ts \pd_{a,m}(n) \ts t^n \. = \. \prod_{r=0}^\infty \ts \left(1+t^{a+rm}\right)\..
$$
In other words, $\pd_{a,m}(n)$ is the number of partitions of~$n$ into distinct parts
in arithmetic progression $R=\{a,a+m,a+2m,\ldots\}$. It is known~\cite{RS} that
$$\pd_{a,m}(n+1) \. > \. \pd_{a,m}(n) \. > \. 0\ts,
$$
for all $n$ large enough. Below we present a stronger result.

\medskip
\subsection{Partitions into finite arithmetic progressions}\label{ss:basic-pdp-fin} \.
Denote by $R=R(a,m,k)=\{a,a+m,a+2m,\ldots,a+km\}$ a finite arithmetic progression,
with $a,m \ge 1$, such that \ts gcd$(a,m)=1$ as above.  Denote by $\pd_R$
the coefficients in
$$
\sum_{n=0}^N \ts \pd_R(n) \ts t^n \. = \.
\prod_{r=0}^k \ts  \left(1+t^{a+rm}\right) ,
$$
where $N=(k+1)a+\binom{k+1}{2}m$ is the largest degree with a nonzero coefficient.
Note that the sequence $\{\pd_R(n)\}$ is symmetric:
$$
\pd_R(n) \, = \, \pd_R(N-n)\ts.
$$
The following special case of a general result by Odlyzko and Richmond~\cite{OR}
is the key tool we use throughout the paper.

\begin{thm}[\cite{OR}] \label{t:pdp-fin}
For every $R=R(a,m,k)$ as above, there exists $L=L(a,m)$ such that
$$\pd_R(n+1) \. > \. \pd_R(n) \. > \. 0 \ts, \quad \text{for all}
\ L \le n < \lfloor N/2\rfloor \ts.
$$
\end{thm}

In other words,
$$
\pd_R(L)  <  \ldots < \pd_R\left(\frac{N}2-1\right)
<  \pd_R\left(\frac{N}2\right) >  \pd_R\left(\frac{N}2+1\right) > \ldots  >  \pd_R(N-L)
$$
for even~$N$, and
$$
\pd_R(L)  < \ldots < \pd_R\left(\frac{N-1}{2}\right)
= \pd_R\left(\frac{N+1}{2}\right) > \ldots  >  \pd_R(N-L)
$$
for odd~$N$.  Note that~$k$ in the theorem has to be large enough to ensure
that $L < N/2$; otherwise, the theorem is trivially true (there is no such~$n$).

\bigskip

\section{Kronecker products and characters}\label{s:char}

\subsection{Young diagrams} \.
We assume the reader is familiar with the standard results in combinatorics
and representation theory of the symmetric group (see e.g.~\cite{Mac,Sag,Sta}).
Let us review some notations, definitions and basic results.

We use $[\la]$ to denote Young diagram corresponding to partition~$\la$ and
a \emph{hook length} by $h_{ij}=\la_i +\la_j' - i - j+1$, where $(i,j)\in [\la]$.
We denote by $\ds(\la)$ the \emph{Durfee size}
of~$\la$, i.e.~the size of the main diagonal in~$[\la]$.
Define a \emph{principal hook partition} $\wh\la = (h_{1,1},\ldots,h_{s,s})$,
where $s=\ds(\la)$.  Observe that $\wh\la\vdash n$.

\medskip  \subsection{Rim hook tableaux} \.
We use $\chi^\la[\nu]$ to denote the value of an irreducible character
$\chi^\la$ on the conjugacy class of cycle type $\nu$ of the symmetric group~$S_n$.
For a sequence $\ba=(a_1,\ldots,a_\ell)$, $\ell=\ell(\la)$,
denote by $\RH(\la,\ba)$ the set of \emph{rim hook tableaux} of shape~$\la$
and weight~$\ba$, with rim hooks $h_i$ of size $|h_i|=a_i$.
The \. $\sign(A)$ of a tableaux $A\in \RH(\la,\ba)$ is the product
of $(-1)^{\ell(h_i)-1}$ over all rim hooks $h_i\in A$.  The
\emph{Murnaghan--Nakayama rule} then says that for every permutation~$\ba$
of~$\nu$,
$$\chi^\la[\nu] \. = \. \sum_{A\in \RH(\la,\ba)} \. \sign(A)\ts.
$$
More generally, the result extend verbatim to skew shapes $\la/\mu$
(see e.g.~\cite[$\S 7.17$]{Sta}).

\medskip  \subsection{The Frobenius and Giambelli formulas} \.
Recall the \emph{Frobenius formula} for the character~$\chi^\la$,
$\ell(\la) =2$~:
$$
\chi^{(n-\ell,\ell)} \. = \. \chi^{(n-\ell)\circ (\ell)} \. - \. \chi^{(n-\ell+1)\circ (\ell-1)}.
$$
where $\la \circ \mu$ is a skew partition as in the Figure~\ref{f:circle}.
\begin{figure}[hbtp]
   \includegraphics[scale=0.56]{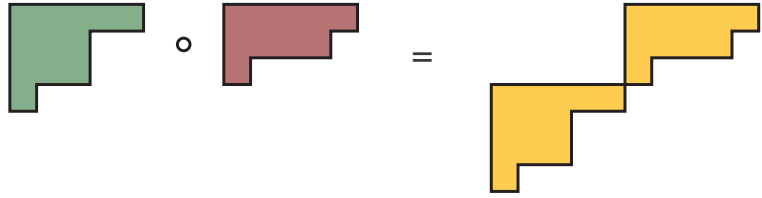}
   \caption{Partitions $\la$, $\mu$, and $\la\circ\mu$, where $\la=(5,3,3,1)$, $\mu=(5,4,1)$}
   \label{f:circle}
\end{figure}

Similarly, the \emph{Giambelli formula} for the character $\chi^\la$, where $\ds(\la) =2$~:
$$
\chi^{(a_1+1,a_2+2,2^{b_2},1^{b_1-b_2-1})} \. = \.
\chi^{(a_1+1,1^{b_1})\circ (a_2+1,1^{b_2})} \. - \.
\chi^{(a_1+1,1^{b_2})\circ (a_2+1,1^{b_1})}.
$$
where $n=a_1+a_2+b_1+b_2+2$.  The formula is illustrated in Figure~\ref{f:giam}
(here $a_1=8$, $a_2=2$, $b_1=5$, and $b_2=3$).

\begin{figure}[hbtp]
  \includegraphics[scale=0.42]{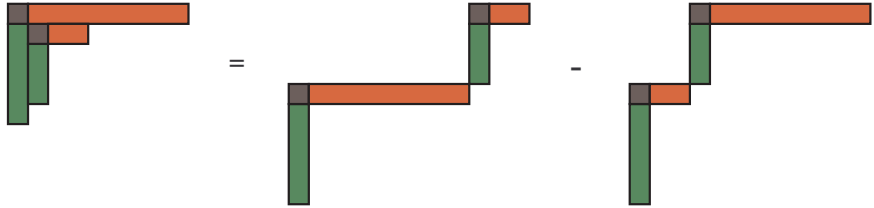}
   \caption{Partitions $\la=(9,4,2^3,1)$, $(9,1^5)\circ (3,1^3)$, and
   $(3,1^5)\circ (9,1^{3})$.}
   \label{f:giam}
\end{figure}

\medskip
\subsection{Kronecker products}\.  Let $\la, \mu \vdash n$.
The \emph{Kronecker product} of characters $\chi^\la$ and $\chi^\mu$ satisfies
$(\chi^\la\otimes\chi^\mu)[\nu] = \chi^\la[\nu]\cdot\chi^\mu[\nu]$.
\emph{Kronecker coefficients} are defined as
$$g(\la,\mu,\nu) \. := \. \<\chi^\la\otimes\chi^\mu, \chi^\nu\>
\. = \. \<1, \chi^\la\otimes\chi^\mu \otimes \chi^\nu\>\ts,$$
and thus they are symmetric for all $\la, \mu, \nu \vdash n$.

For a partition $\la\vdash n$, denote by $\Ups(\la)$ the set of
$\mu \vdash n$ such that $g(\mu,\la,\la)>0$.  The tensor product
conjecture says that $\Ups(\la)=P_n$ for some~$\la$.  The Saxl
conjecture says that $\Ups(\rho_k)=P_n$.  We call $\rho_k$ the
\emph{staircase shape} of order~$k$.

\medskip
\subsection{Chopped square}\.
Denote $\eta_k=(k^{k-1},k-1)\vdash k^2-1$. We call $\eta_k$
the \emph{chopped square shape} of order~$k$.  Obviously,
$\ds(\eta_k)=k-1$, and the scaled limit shape is a $1/2 \times 1/2$ square.

\begin{conj}  \label{conj:chop}
For all $k\ge 2$, we have $\Ups(\eta_k)=P_n$.
\end{conj}

This conjecture was checked by Andrew Soffer for $k\le 5$.

\medskip  \subsection{Caret shape}\.
Consider
$$\mon_k \. = \. (3k-1,3k-3,\ldots,k+3,k+1,k,k-1,k-1,k-2,k-2,\ldots,2,2,1,1)\ts,$$
which
we call a \emph{caret shape}.
Note that $\mon_k'=\mon_k$, $n=|\mon_k|=3k^2$, $\ds(\mon_k)=k$,
and the principal hook partition
$\wh{\mon_k}=(6k-3,6k-9,\ldots,3)$.  After $1/\sqrt{n}$ scaling, the partition
has a 4-gon limit shape as in Figure~\ref{f:3shape}.

\begin{figure}[hbtp]
   \includegraphics[scale=0.345]{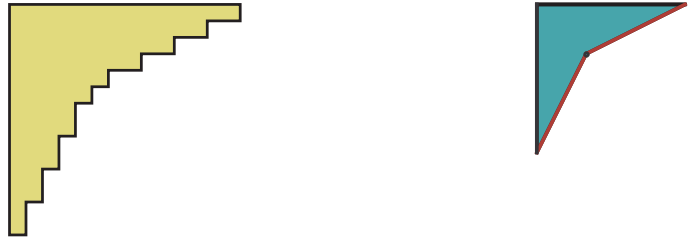}
   \caption{Partition $\mon_5$ and the limit shape of $\mon_k$.}
   \label{f:3shape}
\end{figure}

\begin{conj}  \label{conj:ups}
For all $k\ge 2$, we have $\Ups(\mon_k)=P_n$.
\end{conj}

\begin{rem}{\rm
Despite the less elegant shape of~$\mon_k$ partitions,
there seems to be nearly as much evidence in favor of 
Conjecture~\ref{conj:ups} as in favor of the Saxl Conjecture~\ref{conj:saxl} 
(see corollaries~\ref{c:ups-hook} and~\ref{c:ups-two}).
See Section~\ref{s:rand} for more caret shapes and possibility
of other self-conjugate shapes satisfying the Tensor Product
Conjecture.
}
\end{rem}
\bigskip


\section{Known results and special cases}\label{s:known}

\subsection{General results}\.
The following results are special cases of known results about
Kronecker products, applied to our case.

\begin{lemma} \label{l:basic}
Let $\nu\vdash n$. Then $\chi^\mu \otimes \chi^{(n)} = \chi^\mu$ and
$\chi^\mu \otimes \chi^{(1^n)} = \chi^{\mu'}$.
\end{lemma}

\begin{proof}
Note that $\chi^{(n)}$ is the trivial character and $\chi^{(1^n)}$ is the sign character.
The first identity is trivial; the second follows from the Murnaghan--Nakayama rule.
\end{proof}

\begin{cor}
Let $\mu \vdash n$.
Then $\mu = \mu'$ if and only if $(1^n) \in \Ups(\mu)$.
\end{cor}

\begin{proof}
By Lemma~\ref{l:basic}
$$ \< \chi^\mu \otimes \chi^\mu, \chi^{(1^n)} \> =
\< \chi^\mu, \chi^\mu \otimes \chi^{(1^n)} \> =
\< \chi^\mu, \chi^{\mu'} \>.$$
So, the claim follows.
\end{proof}

\begin{prop}\label{p:sign}
Let $\mu \vdash n$.
If $\mu= \mu'$, then $g(\la,\mu,\mu)=g(\la',\mu,\mu)$.
\end{prop}

\begin{proof}
Simply observe that by Lemma~\ref{l:basic} we have
$$\aligned
g(\la,\mu,\mu) \, & = \, \<\chi^\la, \chi^{\mu'}\otimes \chi^\mu\> \,
 = \, \<\chi^\la, \chi^\mu \otimes \ts \chi^{(1^n)} \ts   \otimes \chi^\mu\> \\
& = \, \<\chi^\la, \ts \chi^{(1^n)} \ts   \otimes \chi^\mu \otimes \chi^\mu\>
 = \, \<\chi^\la\otimes \ts \chi^{(1^n)} \ts ,\chi^{\mu'} \otimes \chi^\mu\> =
g(\la',\mu,\mu),
\endaligned
$$
as desired. \end{proof}

\begin{thm}[\cite{BR}]  \label{t:br}
If $\ds(\la)>2\ts \ds(\mu)^2$, then $\la \notin \Ups(\mu)$.\footnote{Thm.~3.26
in~\cite{BR} is more general, and applies to all diagonal lengths of
constituents in all $\chi^\la\otimes\chi^\nu$.}
\end{thm}

\begin{cor}
If $\ds(\mu)<  \frac{n^{1/4}}{\sqrt{2}}$,
then $\Ups(\mu)\ne P_n$.
\end{cor}

Therefore, if $\mu \ne \mu'$ or $ 2 \ds(\mu)^2 < \sqrt{n}$,
then character $\chi^\mu$ cannot be used in the tensor product conjecture.

\begin{thm}[\cite{BB}]\label{t:BB}
If $\la=\la'$, then $\la \in \Ups(\la)$, i.e. $g(\la,\la,\la)>0$.
\end{thm}

\smallskip

\begin{rem} {\rm
Note that neither of the results in this section disproves the Saxl conjecture,
nor conjectures~\ref{conj:chop} and~\ref{conj:ups}.  Indeed, all these
partitions are self-conjugate and have Durfee size of the order
$\Theta(\sqrt{n})$.
}
\end{rem}

\smallskip

\subsection{Large~$\mu_1$} \.
Let $\nu$ be a composition and denote by $\vv_\nu(\mu)$ the number of
ways to remove ribbon hook shaped~$\nu$ from~$\mu$ such that $\mu/\nu$
is a Young diagram.
The following result is given in~\cite{Val3}.
See also~\cite{Sax, Zis}.

\begin{thm} \label{t:val-corners}
For $\la=(n-r,\tau)$ and $\tau\vdash r$, denote
$f(\tau,\mu)=g(\la,\mu,\mu)$. Then:
$$\aligned
f(\emp) & = 1 \. , \qquad
f(1,\mu) = \vv_1(\mu)-1 \. , \qquad f(1^2,\mu)  = \bigl(\vv_1(\mu)-1\bigr)^2 \. ,\\
f(2,\mu) &= \vv_2(\mu)+\vv_{1^2}(\mu)+\vv_1(\mu)^2-2\vv_1(\mu) \. ,\\
f(3,\mu)& = \vv_3(\mu)+\vv_{1^3}(\mu)+\vv_{21}(\mu)+\vv_{12}(\mu)+ \bigl(2\vv_1(\mu)-3\bigr)\bigl(\vv_2(\mu)+\vv_{1^2}(\mu)\bigr) \\
& \qquad +\vv_1(\mu)^3-4\vv_1(\mu)^2+3\vv_1(\mu)\. ,\\
f(21,\mu)& = \vv_{21}(\mu)+\vv_{12}(\mu)+ \bigl(3\vv_1(\mu)-4\bigr)\bigl(\vv_2(\mu)+\vv_{1^2}(\mu)\bigr)\\
& \qquad +2\vv_1(\mu)^3-8\vv_1(\mu)^2+7\vv_1(\mu)\. ,\\
f(1^3,\mu)& = \vv_{21}(\mu)+\vv_{12}(\mu)+ \bigl(\vv_1(\mu)-1\bigr)\bigl(\vv_2(\mu)+\vv_{1^2}(\mu)\bigr)\\
& \qquad +\vv_1(\mu)^3-4\vv_1(\mu)^2+4\vv_1(\mu)-1.
\endaligned$$
\end{thm}

Calculating these values explicitly, gives the following result:

\begin{cor} \label{c:small}
Let $\mu \vdash n$.
If $\mu = \mu'$ and $\mu$ is not a square, we have:
$$(n), \,(n-1,1), \,(n-2,2),\,(n-2,1^2), \, (n-3,3), \,(n-3,2,1),\,(n-3,1^3) \in \Ups(\mu).$$
However, $(n-1,1), \. (n-2,1^2) \notin \Ups(k^k)$ for $n=k^2$.
\end{cor}

In other words, the theorem rules out square partitions in the tensor product conjecture.

\begin{proof}
For $n=3$ the statement follows from a direct calculation, so we assume $n\ge 4$.
Since $\mu$ is not a square, $\vv_1(\mu) \ge 2$.
So, we have that $(n)$, $(n-1,1)$ and $(n-2,1^2)$ are in $\Ups(\mu)$.
Let us consider first the case $\vv_1(\mu) =2$.
If $\mu$ is not a hook, since $\mu = \mu'$,  one has either $\vv_2(\mu) + \vv_{1^2}(\mu) \ge 4$
or $\vv_2(\mu) + \vv_{1^2}(\mu) =2$ and $\vv_{21}(\mu)=1$; so, the remaining partitions are in $\Ups(\mu)$.
If $\mu$ is a hook, since $\mu = \mu'$,  one has $\vv_2(\mu) + \vv_{1^2}(\mu) =2$, which
implies that $(n-2,2)$, $(n-3,2,1)$ and $(n-3,1^3)$ are in $\Ups(\mu)$.
For $(n-3,3)$, we must have $n\ge 6$, thus $\vv_3(\mu) + \vv_{1^3}(\mu) =2$, and we obtain that
$(n-3,3)$ is in $\Ups(\mu)$.
Finally, if $\vv_1(\mu)\ge 3$ all the polynomials in $\vv_1$ at the end of the summations are nonnegative,
and we have either $\vv_2(\mu) + \vv_{1^2}(\mu) \ge 2$ or $\vv_{21}(\mu) \ge 1$.
So, the claim follows.
\end{proof}

\begin{ex}{\rm
A direct calculation gives
$$\chi^{(2^2)} \otimes \chi^{(2^2)} \.  = \. \chi^{(4)} + \chi^{(2^2)} +
\chi^{(1^4)}\ts.
$$
In other words, $\Phi(2^2)$ is missing only $(3,1)$ and $(2,1^2)$.
}
\end{ex}

\medskip  \subsection{Two rows}\. There are many results on Kronecker product of characters with at least
one partition with two row (see e.g.~\cite{BO1,BO,BOR,RW,Ros}).  Of these, only the result of
Ballantine and Orellana~\cite{BO} extends to general partitions.  Its statement is rather
technical, so we instead present a direct corollary from it adapted to our situation.

\begin{thm} \label{t:BO}
For every $\mu \vdash n$, $2\le p \le \min\{\ell(\mu), \frac{1+\mu_1}{2}\}$,
we have $(n-p, p) \in \Ups(\mu)$.
\end{thm}

\begin{proof}  In the notation of~\cite[Thm.~3.2]{BO}, let $\al=(1^p) \ssu \mu$, and consider
tableau~$T$ of shape $(\mu/1^p)$ filled with numbers~$\mathbf{i}$ in $i$-th row.  This gives a
\emph{Kronecker tableaux} of type $(\mu/1^p)$ counted in the combinatorial interpretation
in~\cite{BO}.  Thus, for $\la=(n-p,p)$, we have $g(\la,\mu, \mu)\ge 1$, as desired.
\end{proof}

For example, for $\mu=\rho_k$, we have $\ell(\mu)=\mu_1=k$, and the result
gives positive Kronecker coefficients $k\bigl((n-p,p),\mu,\mu\bigr)$ for $p\le (k+1)/2$.
Unfortunately, for larger $p$ the result in~\cite{BO} gives only an upper bound, while
we need a lower bound.  In Section~\ref{s:two}, we improve the above bound to~$p\le n/2$.

\medskip
\subsection{Hooks}\. Of the extensive literature, Blasiak's combinatorial
interpretation (see~\cite{Bla}, Theorem~3.5), is perhaps the most convenient.
Again, the statement is rather technical, so we instead present
the following corollary.

\begin{thm} \label{t:Bla}
Let $\mu=(\mu_1,\ldots,\mu_\ell)$ be a partition of~$n$, such
that $\mu_1> \ldots >\mu_{r}$ for some $r\le \ell$. Then:
$$
(n-m,1^m) \in \Ups(\mu) \quad \text{for all} \ \. m< r\ts.
$$
\end{thm}

\begin{proof}  In notation of~\cite{Bla}, consider tableau~$T$ of shape~$\mu$,
filled with numbers~$\mathbf{i}$ in $i$-th row.  Now place a bar on the first $m$ integers
in the first column.  Denote by~$w$ the word obtained by reading unbarred skew
shape from right to left, and then barred shape from left to right (in this case, just
the first column from bottom to top).  We have:
$$
w \. = \. 1^{\mu_1-1} 2^{\mu_2-1} \ldots \ts m^{\mu_m-1} (m+1)^{\mu_{m+1}}  \ldots \ts
\ell^{\mu_\ell} \ts m \ts (m-1) \ts \ldots \ts 2 \ts 1
$$
Now observe the inequalities in the statement imply that word~$w$ is a
\emph{ballot sequence}.  In the language of~\cite{Bla}, this implies that
the reverse of~$w$ is a \emph{Yamanouchi word} of content~$\mu$.

Now consider tableau $C(T)$ obtained when~$T$ is converted into natural order.
Note that~$C(T)$ will have its lower left corner $(\ell,1)$ unbarred, because
in the $\ell$-th row, tableau~$T$ has only numbers~$\ell$ which are unbarred
and larger than all barred numbers, and thus do not move during conversion.
Therefore, tableau~$T$ gives the desired tableau in the combinatorial
interpretation of $g(\la,\mu,\mu)$.  \end{proof}

\medskip

\subsection{Large Durfee size}\label{ss:ds} \.
The following result is well known and easy to prove.

\begin{lemma}\label{l:large-durfee}
We have $\chi^\mu\bigl[\,\wh\mu\,\bigr]= \pm 1$, for all $\mu \vdash n$.
Moreover, if $\mu=\mu'$, then
$$\chi^\mu\bigl[\,\wh\mu\,\bigr] \. = \. (-1)^{(n-\ds(\mu))/2}.
$$
\end{lemma}

\begin{proof} The principal hook condition implies that there is a unique
rim hook condition in the Murnaghan--Nakayama rule. The second part follows
by taking the product of signs of all hooks.
\end{proof}

\begin{prop} \label{cor:ds-exp}
We have: \, $|\Ups(\rho_k)|\ts > \ts 3^{\lceil k/2\rceil-1}$
\ts and \ts $|\Ups(\mon_k)|\ts > \ts 5^{k-1}$.
\end{prop}

\begin{proof}
There are two sequences of principal hook partitions for $\rho_k$:
For $k=2m+1$ odd the sequence is
$ (4m+1, 4m-3, \dots, 5,1)$.
For $k=2m$ even the sequence is
$(4m-1, 4m-5, \dots, 7, 3)$.
In each case the proof is by induction on $m$.
We show the even case. The odd one is similar.
For each $\la \in \{(3), (2,1), (1^3) \}$ there is exactly one rim hook tableau
of shape $\la$ and weight $(3)$.
So that, by the Murnaghan--Nakayama rule and the Main Lemma, $|\Ups(\rho_2)| = 3 >1$.
We assume, by induction hypothesis, that there are $3^m$ partitions $\la$ such that
for each of them there is exactly one rim hook tableau of shape $\la$ and weight
$(3, \dots, 4m-1)$, and that the rim hook of size $4m-1$ intersects the first row and
the first column of $[\la]$.
For each such $\la$ we construct three partitions as follows:
Let $H$ be the rim hook in $\la$ of size $4m-1$ with end boxes $(1,a)$ and $(b,1)$.
Define
$$
\widetilde H = \{ (x+1,y+1) \mid (x,y)\in H \} \cup \{ (1,a+1), (b+1,1)\}.
$$
Then $|\widetilde H| = 4m+1$.
Define partitions of size $(4m+3) + (4m-1) + \cdots$ by
\begin{align*}
[\la(1)] & = [\la] \cup \widetilde H \cup \{ (1, a+2), (1, a+3) \}; \\
[\la(2)] & = [\la] \cup \widetilde H \cup \{ (1, a+2), (b+2, 1) \}; \\
[\la(3)] & = [\la] \cup \widetilde H \cup \{ (b+2, 1), (b+3, 1) \}.
\end{align*}
We claim that for each $i=1,2,3$, $\la(i)$ has exactly one rim hook tableau
of shape $\la(i)$ and weight $(3, 7, \dots, 4m+3)$.
The southeast border of $[\la(i)]$ is exactly a rim hook $H(i)$ of size
$4m+3$.
It intersects the first row and the first column of $[\la(i)]$.
By construction $[\la(i)] \setminus H(i) = [\la]$.
But, by induction hypothesis, there is only one rim hook tableau of
shape $\la$ and weight $(3, 7, \dots, 4m-1)$.
So, there is only one rim hook tableau of shape $\la(i)$ and weight
$(3,7, \dots, 4m+3)$.
So, by the Main Lemma, $\la(i) \in \Ups (\rho_{2m+2}) $.

It remains to show that the $3^{m+1}$ partitions just constructed are all
different.
This follows also by induction and the fact that
the construction of $\lambda(i)$ from $\lambda$ is reversible, since $\lambda(i)$ has exactly $3-i$ parts of size 1.
Thus, $|\Ups(\rho_{2m+2})| \ge 3^{m+1} > 3^m$.

For the caret shapes there is only one sequence of principal hook partitions
$\wh \gamma_k = (6k-3, 6k-9, \dots, 3)$.
Since $\ds(\mon_k)=k$, a similar argument now proves the second claim.
\end{proof}

In other words, in both cases the number of irreducible constituents
is weakly exponential $\exp\Theta(\sqrt{n})$.  Indeed, the corollary
gives the lower bound and the asymptotics for $\pi(n)$ gives the
upper bound. Note also that the lemma gives nothing for the chopped
square shape~$\eta_k$.

\medskip  \subsection{Large principal hooks}\label{ss:php} \.
The following result is a trivial consequence of the classical Murnaghan--Nakayama rule.

\begin{lemma} Suppose $\la, \mu \vdash n$ and $\wh\la_1 < \wh \mu_1$.
Then $\chi^\la[\widehat{\mu}]=0$.
\end{lemma}

From here we conclude the following counterpart of Proposition~\ref{cor:ds-exp}.

\begin{prop} \label{p:php-exp}
There are at least $3^{\lceil k/2\rceil-3}$ partitions $\la$ of $n=k(k+1)/2$
such that $\chi^\la[\wh \rho_k]=0$.  Similarly, there are at least $5^{k-3}$
partitions $\la$ of $n=3k^2$ such that $\chi^\la[\wh \mon_k]=0$.
\end{prop}

\begin{proof}
Follow the construction as in the proof of Proposition~\ref{cor:ds-exp},
to construct $3^{\lceil k/2\rceil-3}$ partitions $\la$ with principal hooks of size
$$(4\ts m-1, 4\ts m-3, 4\ts m-7,4\ts m-11, \dots, 5,3) \quad \text{for} \ \  k=2m+1, \ \ \  \text{and}
$$
$$(4\ts m-3, 4\ts m-5,4\ts m-9,4\ts m-13,\dots, 7,5) \quad \text{for} \ \  k=2\ts m\ts.
$$
Here $3^2$ possibilities are lost when counting placements of
the outer and the inner rim hooks.
By the lemma above, all such characters $\chi^\la[\wh \rho_k]=0$.   The second
part follows verbatim.
\end{proof}

The proposition implies that there is a weakly exponential number of
partitions for which the Saxl conjecture and Conjecture~\ref{conj:ups}
cannot be proved.  Curiously, this approach does not apply to~$\eta_k$.
In Subsection~\ref{ss:rand-caret} we prove a much stronger result about the number
of partitions~$\la$ such that $\chi^\la[\wh \mon_k]=0$.

\smallskip

\begin{ex}{\rm
Of course, just because $\chi^\la[\wh \mu]=0$ it does not mean that $\la \notin \Phi(\mu)$.
For example $\chi^{(5,1)} \in \Phi(\rho_3)$, even though
$$\chi^{(5,1)}[\wh \rho_3] \. = \. \chi^{(5,1)}[5,1] \. = \. 0\ts.
$$
}\end{ex}

\bigskip


\section{Hooks in tensor squares}\label{s:hook}

\subsection{Chopped square shape}  \. Let $n=k^2-1$, so that $\eta_k=(k^{k-1},k-1)\vdash n$.
Recall that $\wh\eta_k = (2k-1,2k-3,\ldots,7,5,3)$.

\begin{lemma} \label{l:eta-hook}
There exists a constant~$L$, s.t.~$(n-\ell,1^\ell) \in \Ups(\eta_k)$,
for all $L\le \ell < n/2$.
\end{lemma}

\begin{proof}[Proof of Lemma~\ref{l:eta-hook}]
By the Main Lemma (Lemma~\ref{t:main}), it suffices to show that
$$\chi^{(n-\ell,1^\ell)}[\wh \eta_k] \. > 0 \quad \text{for $\ell$ large enough.}
$$
 We claim that the above character is equal to
$$
(\ast) \qquad \pd_{R}(\ell)-\pd_{R}(\ell-1)+\pd_{R}(\ell-2)
$$
where $R=\{5,7,\ldots,2k-1\}$
(see $\S$\ref{ss:basic-pi} for notations).  By Theorem~\ref{t:pdp-fin},
$$
\pd_R(\ell)\. >\. \pd_R(\ell-1)\. > \. \pd_R(\ell-2)\. > \. 0
$$
for $\ell\le n/2$ large enough, this would prove the theorem.

For~$(\ast)$, by the Murnaghan--Nakayama rule, the character
is equal to the sum over all rim hook tableaux of shape
$(n-l,1^\ell)$ and weight $\wh\eta_k$  of the sign of the tableaux.
For convenience, order the parts of $\wh\eta_k$ in increasing order.
Note that the sign of every rim hook which fits inside in either leg or arm of
the hook is positive.  There are 3 ways to place a 3-hook, with
the foot of size $\ell$, $\ell-1$ and $\ell-2$, respectively.
Therefore, the number of rim hook tableaux is equal to
the number of partitions into distinct parts in~$R$, as in~$(\ast)$.
\end{proof}

\begin{figure}[hbtp]
   \includegraphics[scale=0.46]{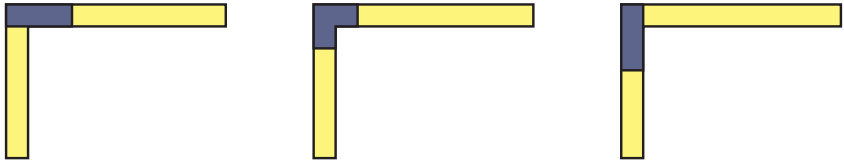}
   \caption{Three ways to place a 3-hook into a hook diagram.}
   \label{f:3rim-hook}
\end{figure}

\begin{ex}\label{ex:const}
{\rm  Although we made no attempt to find constant~$L$ in the lemma,
we know that it is rather large even for $k\to \infty$.  For example,
$$
\pd_{5,2}(21) \ts - \ts \pd_{5,2}(20) \ts + \ts \pd_{5,2}(19)
\. = \. 0,
$$
which implies $\chi^{(n-21,1^{21})}[\wh\eta_k] = 0$ for all $n \ge 21$, $k\ge 11$.
Note also that function $\pd_{5,2}(n)$ continues to be non-monotone for larger~$n$,
i.e. $\pd_{5,2}(41)=15$ and $\pd_{5,2}(42)=14$.
}
\end{ex}

\medskip

\subsection{Caret shape}  \. Let $n = 3k^2$, so that $\gamma_k \vdash n$.
Recall that $\wh\gamma_k = (6k-3, 6k-9, \dots, 9,3)$.
The following is the analogue of Lemma~\ref{l:eta-hook} for the caret shape.

\begin{lemma} \label{l:ups-hook}
There exists a constant~$L$, such that $(n-\ell,1^\ell) \in \Ups(\mon_k)$,
for all $L\le \ell \le n/2$.
\end{lemma}

Combined with Theorem~\ref{t:Bla} as above, we obtain:

\begin{cor}\label{c:ups-hook}
For $k$ large enough, we have $\ts (n-\ell,1^\ell) \in \Ups(\mon_k)$,
for all $1\le \ell \le n-1$.
\end{cor}

\begin{proof}[Proof of Corollary~\ref{c:ups-hook}]
Let $k>L$.  Theorem~\ref{t:Bla} proves the case $\ell \le k-1$ and the
lemma gives $L\le \ell \le n/2$.  In total, these cover all
$0\le \ell \le n/2$.  Now Proposition~\ref{p:sign} prove the remaining
cases $n/2 < \ell \le n-1$.  \end{proof}

\begin{proof}[Proof of Lemma~\ref{l:ups-hook}]
The proof follows the argument in the proof of Lemma~\ref{l:eta-hook}.
The difference is that the removed 3-rim hook can be removed only one way,
as the other two values are not zero mod~3.  This simplifies the character
evaluation and gives
$$
\bigl|\chi^{(n-\ell,1^\ell)}[\wh \mon_k]\bigr| \. = \.  \pd_R\bigl(\lfloor \ell/3\rfloor\bigr),
$$
where $R=\{3,5,\ldots, 2k-1\}$ is obtained from~$\wh \mon_k$ by removing
the smallest part and then dividing by~3. Thus, the above character is nonzero,
and the Main Lemma implies the result.
\end{proof}

\medskip

\subsection{Staircase shape}  \.
Let $n=\binom{k+1}{2}$, so that $\rho_k\vdash n$.
The following result is the analogue of Lemma~\ref{l:eta-hook} for the staircase shape.
Note that $\wh \rho_k=(2k-1, 2k-5, 2k-9,\ldots)$. There are two
different cases: odd~$k$ and even~$k$, which correspond to the
smallest principal hooks $(\ldots, 9,5,1)$ and $(\ldots,11,7,3)$,
respectively.

\begin{lemma} \label{l:rho-hook}
There exists a constant~$L$, such that $(n-\ell,1^\ell) \in \Ups(\rho_k)$,
for all $L\le \ell < n/2$.
\end{lemma}

Combined with Theorem~\ref{t:Bla} and Proposition~\ref{p:sign}, we obtain:

\begin{cor}
For $k$ large enough, we have $(n-\ell,1^\ell) \in \Ups(\rho_k)$,
for all $1\le \ell \le n-1$.
\end{cor}

The proof of the corollary follows verbatim the proof of Corollary~\ref{c:ups-hook}.

\begin{proof}[Proof of Lemma~\ref{l:rho-hook}]
Let $\la=(n-\ell,1^\ell)$ as above.  Treat the even~$k$ case in the
same way as the of Lemma~\ref{l:eta-hook} above.  Again, there are
three ways to remove 3-rim hook and  Theorem~\ref{t:pdp-fin} implies
that the character is nonzero.  The Main Lemma now implies the result.

The odd~$k$ case is even easier: hook~1
can be placed in~$[\la]$ in a unique way, after which we get
$\ts \chi^\la[\wh \rho_k]$ is the number of partitions $\pd_R(\ell)$
into distinct parts $R=\{5,9,13,\ldots,2k-1\}$. \end{proof}

\bigskip


\section{Two row shapes in tensor squares}\label{s:two}

\subsection{Chopped square shape}  \. Let $\eta_k=(k^{k-1},k-1)\vdash n$,
$n=k^2-1$, be as above.

\begin{lemma} \label{l:eta-two}
There exists a constant~$L$, such that $(n-\ell,\ell) \in \Ups(\eta_k)$,
for all $\ts L\le \ell \le n/2$.
\end{lemma}

This immediately gives:

\begin{cor}\label{c:eta-two}
For $k$ large enough, we have $(n-\ell,\ell) \in \Ups(\eta_k)$,
for all \. $0\le \ell \le n/2$.
\end{cor}

\begin{proof}[Proof of Corollary~\ref{c:eta-two}]
By Lemma~\ref{l:basic} and the symmetry of Kronecker coefficients, we have $\ell=0$ case.
By Theorem~\ref{t:BO}, we have the result for $\ell \le k/2$. Finally,
the lemma gives $L \le \ell \le n/2$ case.  Taking $k\ge 2\ts L$, completes the proof.
\end{proof}

\begin{proof}[Proof of Lemma~\ref{l:eta-two}] \.
Recall the Frobenius formula
$$
\chi^{(n-\ell,\ell)} \. = \. \chi^{(n-\ell)\circ (\ell)} \. - \. \chi^{(n-\ell+1)\circ (\ell-1)}.
$$
By the Murnaghan--Nakayama rule for skew shapes, we have:
$$
\chi^{(n-m)\circ (m)}[\wh\eta_k]\. =\. \pd_R(m)\ts,
$$
where $R=\{3,5,\ldots,2k-1\}$.  Therefore, for $L \le \ell \le n/2$, by
Theorem~\ref{t:pdp-fin} we have
$$
\chi^{n-\ell,\ell}[\wh \eta_k] \. = \. \pd_R(\ell)\. - \. \pd_R(\ell-1)\ts > \ts 0\ts.
$$
Now the Main Lemma implies the result.  \end{proof}

\medskip  \subsection{Staircase shape}  \. Let $\rho_k=(k,{k-1},\ldots, 1)\vdash n$,
$n=k(k+1)/2$, be as above.

\begin{lemma} \label{l:rho-two}
There exists a constant~$L$, such that for all $L<\ell \le n/2$
we have $(n-\ell,\ell) \in \Ups(\rho_k)$.
\end{lemma}

There are two cases to consider: even~$k$ and odd~$k$.  Each case
follows verbatim the proof of Lemma~\ref{l:eta-two}.  We omit the
details.  Combined with Theorem~\ref{t:BO}, this immediately gives:

\begin{cor}
For $k$ large enough, we have $(n-\ell,\ell) \in \Ups(\rho_k)$,
for all \. $0\le \ell \le n/2$.
\end{cor}

\medskip

\subsection{Caret shape}  \. Let $\mon_k=(3k-1,3k-3\ldots,2^2,1^2)\vdash n$,
$n=3\ts k^2$, be the caret shape defined above.

\begin{lemma} \label{l:ups-two}
There exists a constant~$L$, such that $(n-\ell,\ell) \in \Ups(\mon_k)$,
for all $\ell = 0,1$~{\rm mod}~$3$, $L\le \ell \le n/2$.
\end{lemma}

Again, combined with Theorem~\ref{t:BO}, this immediately gives:

\begin{cor}\label{c:ups-two}
For $k$ large enough, we have $(n-\ell,\ell) \in \Ups(\mon_k)$,
for all \. $0\le \ell \le n/2$, $\ell =0,1$~{\rm mod}~$3$.
\end{cor}

\begin{proof}[Proof of Lemma~\ref{l:ups-two}] \. Recall the Frobenius formula
$$
\chi^{(n-\ell,\ell)} \. = \. \chi^{(n-\ell)\circ (\ell)} \. - \. \chi^{(n-\ell+1)\circ (\ell-1)}.
$$
By the Murnaghan--Nakayama rule for skew shapes, each of the characters
on the right is equal to the number of partitions $\pd_R(\ell)-\pd_R(\ell-1)$
into distinct parts in $R=\{3, 9, \ldots, 6k-3\}$.
Since only one of $\{\ell,\ell-1\}$ is divisible by~3, we conclude that
$\chi^{n-\ell,\ell}[\mon_k]$ is equal to either $\pd_R(\ell)$ or $-\pd_R(\ell-1)$.
Therefore,
$$
\chi^{n-\ell,\ell}[\mon_k] \. = \. \pm \ts \pd_S\left(\left\lfloor \frac{\ell-1}{3}\right\rfloor\right),
\quad \text{where} \ \. S=\{1,3,\ldots,2k-1\}\ts.
$$
Now Theorem~\ref{t:pdp-fin} and the Main Lemma imply the result.  \end{proof}

\bigskip


\section{Variations on the theme}\label{s:var}

\subsection{Near--hooks} \. Let $\la = (n-\ell-m,m,1^\ell)$, which we call
the \emph{near--hook}, for small $m\ge 2$.
We summarize the results in the following theorem.

\begin{thm}[Near--hooks in staircase shapes] \label{t:near-hook}
There is a constant $L>0$, such that for all $\ell, k \ge L$, $\ell < n/2-5$,
$n=k(k+1)/2$, we have:
$$
(n-\ell-2,2,1^\ell), \ldots, (n-\ell-10,10,1^\ell) \. \in \ts \Ups(\rho_k)
$$
\end{thm}

\begin{proof}[Sketch of proof]
For $\la = (n-\ell-2,2,1^\ell)$, use Giambelli's formula to obtain
$$\chi^{(n-\ell-2,2,1^\ell)} \. = \. \chi^{(n-\ell-2,1^{\ell+1}) \circ (1)} \. - \.
\chi^{(n-\ell-2) \circ (1^{\ell+2})}\ts.
$$
In the odd case, when the first summand of the right side of the equation
is evaluated at $[\wh \rho_k]$ using the Murnaghan--Nakayama rule,
part $(1)$ is placed uniquely, and there are 5 choices for the
5-rim hook.  Thus, the first character evaluated at~$\wh\rho_k$, is equal to
$$
\pd_R(\ell+1)-\pd_R(\ell)+\pd_R(\ell-1)-\pd_R(\ell-2)+\pd_R(\ell-3)\ts,
$$
where $R=\{9,13,\ldots,2k-1\}$.
Similarly, the second character evaluated at~$\wh\rho_k$, is equal to
$$
\pd_R(\ell+2)+\pd_R(\ell+1)+\pd_R(\ell-3)+\pd_R(\ell-4)\ts,
$$
depending on the placement of parts~$(1)$ and~$(5)$.
For the difference, we have

$$\aligned
& -\pd_R(\ell+2)-\pd_R(\ell) +\pd_R(\ell-1)-\pd_R(\ell-2)-\pd_R(\ell-4) \\
& \. = \.
-\bigl[\pd_R(\ell)-\pd_R(\ell-1)\bigr] - \bigl[\pd_R(\ell+2) + \pd_R(\ell-2)+ \pd_R(\ell-4)\bigr] \. < 0\ts,
\endaligned
$$
for $L\le \ell \le n/2-1$, by Theorem~\ref{t:pdp-fin}.  Now the Main Lemma
implies the odd $k$ case.
In the even~$k$ case, there is no part~$(1)$ and the first character is zero;
by Theorem~\ref{t:pdp-fin} the second is positive for~$\ell$ as above, and the proof follows.

For $\la = (n-\ell-3,3,1^\ell)$, use Giambelli's formula to obtain
$$\chi^{(n-\ell-3,3,1^\ell)} \. = \. \chi^{(n-\ell-3,1^{\ell+1}) \circ (2)} \. - \.
\chi^{(n-\ell-3) \circ (2,1^{\ell+1})}\ts.
$$
Since there no part of size~$(2)$, the first character evaluated at~$\wh\rho_k$,
is equal to zero. For the second character, in the odd~$k$ case,
there is a unique way to place 1-hook in the upper left corner of
$(n-\ell-3)$, and then 5-hook in $(2,1^{\ell+1})$, which gives
$$\chi^{(n-\ell-3) \circ (2,1^{\ell+1})} \. = \. -\pd_R(\ell-2)\ts, \quad
\text{where} \ \. R=\{9,13, \ldots\}.
$$
Similarly, in the even~$k$ case, there is a unique way to remove 3-hook
which gives
$$\chi^{(n-\ell-3) \circ (2,1^{\ell+1})} \. = \.-\pd_R(\ell)\ts, \quad
\text{where} \ \. R=\{7,11, \ldots\}.
$$
The rest of the proof follows verbatim.

For $\la = (n-\ell-4,4,1^\ell)$, there is no part~$(3)$ in the odd~$k$,
and the role of odd/even~$k$ are interchanged.  The details are straightforward.
Other results are similar as well. \end{proof}

\begin{rem}{\rm
This sequence of results can be continued for a while, with computations
of multiplicities of $\la = (n-\ell-m,m,1^\ell)$ becoming more complicated
as $m$ grows.  Beyond some point, the naive
estimates as above no longer apply and we need
stronger estimates on the numbers of partitions.
Finally, the characters in cases of chopped square and
the caret shapes can be analyzed in a similar way.
We omit them for brevity.
}\end{rem}

\medskip
\subsection{Near two rows} \.
Let $\la = (n-\ell-m,\ell,m)$, which we call
the \emph{near two rows}, for small $m\ge 1$.
We summarize the results in the following theorem.

\begin{thm}[Near two rows in staircase shapes]  \label{t:near-two}
There is a constant $L>0$, such that for all \ts $\ell \ge L$,
$n=k(k+1)/2$,
we have $(n-\ell-m,\ell,m)  \in \Ups(\rho_k)$ \\
1) \ for $m=1,\ts 5,\ts 7, \ts 8, \ts 9$, \\
2) \ for $m=2,\ts 4$, and $k$~even,\\
3) \ for $m=3$, and $k$~odd.
\end{thm}

\begin{proof}[Sketch of proof]
For $\la = (n-\ell-1,\ell,1)$, use Giambelli's formula to obtain
$$\chi^{(n-\ell-1,\ell,1)} \. = \. \chi^{(n-\ell-1,1^{2}) \circ (\ell-1)} \. - \.
\chi^{(\ell-1,1^2) \circ (n-\ell-1)}\ts.
$$

Again, use skew Murnaghan--Nakayama rule to evaluate the characters at~$\wh \rho_k$,
where the parts are ordered in decreasing order.  Note that each skew partition is
a composition of a hook and a row.  For odd~$k$, part~1 is placed uniquely,
into a single row, and regardless how the rim hook
fits the foot of the hook, it will have positive sign.

We conclude that the difference of character values is equal to
$$ \pd_R(\ell-2) \. - \. \pd_R(\ell+1)\ts, \quad \text{where} \ \ts R=\{5,9,\ldots,2k-1\}, \ k\. - \. \text{odd},
 $$
 $$\pd_R(\ell-1) \. - \. \pd_R(\ell+1)\ts, \quad \text{where} \ \ts R=\{3,7,\ldots,2k-1\}, \ k\. - \. \text{even}.
 $$
Thus, by Theorem~\ref{t:pdp-fin} , in both cases
we have the difference $<0$ for $\ell$ as above.  Use the Main Lemma to obtain $m=1$ case.

The $m=2$ case is similar, with a global change of sign as one of the hooks has even
height. We omit the easy details.

\smallskip

For $\la = (n-\ell-3,\ell,3)$, use Frobenius formula\footnote{In
the context of Schur functions, Frobenius formula is
often called the Jacobi-Trudi identity~\cite{Mac}.}
for $3$ rows~\cite{Sag,Sta}:
$$
\aligned
\chi^{(n-\ell-3,\ell,3)} \. & = \.
\chi^{(n-\ell-3)\circ (\ell) \circ (3)} \. - \. \chi^{(n-\ell-2)\circ (\ell-1) \circ (3)} \\
& - \. \chi^{(n-\ell-3)\circ (\ell+1) \circ (2)} \. + \. \chi^{(n-\ell-1)\circ (\ell-1) \circ (2)} \\
& + \. \chi^{(n-\ell-2)\circ (\ell+1) \circ (1)} \. - \. \chi^{(n-\ell-1)\circ (\ell) \circ (1)}\ts.
\endaligned
$$
Let $k$ be odd.  Evaluated at $\rho_k$, only the last difference is nonzero, giving
$\pd_R(\ell+1)-\pd_R(\ell)$, $R=\{5,9,\ldots\}$, which is positive for large enough~$\ell$.
Similarly, for even~$k$, only the first difference is nonzero, giving
$\pd_R(\ell)-\pd_R(\ell-1)$ which is positive for large enough~$\ell$.
Other cases $m \le 9$ re similar.  We omit the details.
\end{proof}

\begin{rem}{\rm In the theorem we omit three cases (odd $m=2, \ts 4$ and even $m=3$),
since the are no skew rim hook tableaux in that case.
This implies that the corresponding character is zero, and the Main Lemma tells us nothing.  Of course,
one can increase $m$ and/or the number of rows.  Both the Frobenius and the Giambelli formulas become
more involved and it becomes more difficult to compute the characters in terms of partitions into
distinct parts in arithmetic progressions. }
\end{rem}

\bigskip


\section{Random characters}\label{s:rand}

\subsection{Caret shape} \label{ss:rand-caret}
\. We start with the following
curious result.

\begin{prop}\label{p:rand-ups}
Let $n=3k^2$ and $\la \in P_n$ be a random partition.  Then
$\chi^\la[\wh\mon_k]=0$ w.h.p., as $n\to \infty$.
\end{prop}

\begin{proof}
Since all parts of $\wh\mon_k$ are divisible by~3, we have
$\chi^\la[\wh\mon_k]=0$ unless $\la$ has an empty $3$-core.
By the asymptotics given in $\S$\ref{ss:part-as},
the probability of that is
$$\frac{\pi_3(n)}{\pi(3n)} \. = \. O\left(\frac{1}{\sqrt{n}}\right),
$$
as desired.
\end{proof}

\begin{rem}{\rm
The proposition states that almost all character values are zero
in this case.  This suggests that either Conjecture~\ref{conj:ups}
is false for large~$k$, or the Main Lemma is too weak
when it comes to this conjecture even for a constant
fraction of the partitions.
}
\end{rem}

\smallskip

\subsection{Staircase shape} \. Keeping in mind Proposition~\ref{cor:ds-exp}
and Proposition~\ref{p:php-exp}, we conjecture that almost all characters
are not equal to zero:

\begin{conj}\label{conj:rand-rho}
Let $n=k(k+1)/2$ and $\la \in P_n$ be a random partition.  Then
$\chi^\la[\wh\rho_k]\ne 0$ w.h.p., as $n\to \infty$.
\end{conj}

\smallskip

\subsection{Random shapes} \. Note that the random partitions
$\la\vdash n$ are (approximately) self-conjugate and by the
limit shape results (see~$\S$\ref{ss:part-limit}), have
Durfee size
$$\ds(\la) \. \sim \.
\frac{(\ln 2)\sqrt{6\ts n}}{\pi} \. \approx \. 0.54\ts \sqrt{n}\ts.
$$
This raises the
question that perhaps random self-conjugate partitions
satisfy the tensor product conjecture.

\begin{op}\label{op:rand-limit}
Let $\mu \vdash n$ be a random self-conjugate partition of~$n$.
Prove or disprove: \ts $\Ups(\mu) = P_n$ w.h.p., as $n\to \infty$.
\end{op}

The following result is a partial evidence in support of the positive
solution of the problem.

\begin{thm} Let $\mu \vdash n$ be a random self-conjugate partition of~$n$.
Then $\Ups(\mu)$ contains all hooks $(n-\ell,1^\ell)$, $0\le \ell < n$, w.h.p.,
as $n\to \infty$.
\end{thm}

\begin{proof}[Sketch of proof]
We follow the proof of Lemma~\ref{l:eta-hook}. First, recall that self-conjugate
partitions are in natural bijection with partitions into distinct odd parts:
$\mu \leftrightarrow \wh\mu$ (see e.g.~\cite{Pak}).  This implies that
$\ell(\wh\mu) = \Omega(\sqrt{n})$ and the smallest part
$s = \wh\mu_{\ds}=O(1)$~w.h.p., where $\ds=\ds(\mu)$.  Now observe that
the $\gcd(\mu_1,\ldots,\mu_{s})=1$ w.h.p.
Then, by~\cite{OR}, there exists an integer $L$ such that $\pd_R(r)$ are
positive and monotone for $L < r < n-L$, where $R=\{\mu_1,\ldots,\mu_s\}$.
Now, the Murnaghan--Nakayama rule implies that
for $L< \ell n-L$, there are $s=s(\mu)$ ways to remove the smallest part, giving
$$
\chi^{(n-\ell,1^\ell)}[\wh\mu] \. = \.
\pd_R(\ell) - \pd_R(\ell-1) + \pd_R(\ell-2) - \ldots + \pd_R(\ell-s) \, > \ts 0.
$$
Now the Main Lemma gives the result for $\ell$ as above, and for
small $\ell = O(\sqrt{n})$, the result follows from Theorem~\ref{t:Bla}.
\end{proof}

In case Open Problem~\ref{op:rand-limit} is too strong, here
is a weaker, asymptotic version of this claim.

\begin{op}\label{op:rand-limit-approx}
Let $\mu \vdash n$ be a uniformly
 random self-conjugate partition of~$n$.
Prove or disprove:
$$\frac{\bigl|\Ups(\mu)\bigr|}{\bigl| P_n \bigr|} \, \to \. 1 \quad \text{as} \ \.  n\to \infty\ts.
$$
\end{op}

\smallskip

\subsection{Random characters} \.
We believe the following claim closely related to open problems above.

\begin{conj}\label{conj:rand-sc-char}
Let $\la \vdash n$ be a random partitions of~$n$, and
let $\mu$ be a random self-conjugate partition of~$n$.
Then $\chi^\la[\wh \mu]=0$ w.h.p., as $n\to \infty$.
\end{conj}

Heuristically, the first principal hook $\wh\mu_1$ is greater
than $\wh\la_1$ with probability~$1/2$, and by Lemma~\ref{l:large-durfee}
the above character is zero.  If $\ts\wh\mu_1\le \wh\la_1$,
with good probability the $\wh\mu_1$-rim hook can be removed
in a unique way, after which the process is repeated for smaller
principal hooks, in a manner similar to Proposition~\ref{cor:ds-exp}.
This gives that the probability there exists no rim hook
tableaux $\to 0$ as $n\to \infty$.

\smallskip

Now, the above argument is heuristic and may be hard to formalize,
since the ``good probability'' is rather hard to estimate and in principle
it might be close to~0 and lead to many rim hook tableaux.  However,
sharp bound on the distribution of the largest part of the random
partitions (cf.~$\S$\ref{ss:part-limit}) combined with the first
step of the argument can be formalize to prove the following
result.\footnote{The proof will appear elsewhere.}

\begin{prop}\label{prop:rand-sc-char}
Let $\la,\mu \vdash n$ be as in the Conjecture above.
Then there exists $\ve >0$ such that $\chi^\la[\wh \mu]=0$
with probability $<1-\ve$, as $n\to \infty$.
\end{prop}

This implies our Main Lemma is too weak to
establish even the weaker Open Problem~\ref{op:rand-limit-approx},
since a constant fraction of characters are zero.

\smallskip

\subsection{Character table of the symmetric group} \.
Now, Conjecture~\ref{conj:rand-sc-char}
raises a more simple and natural question about
random entries of the character table of~$S_n$.

\begin{conj}\label{conj:rand-char}
Let $\la, \mu \vdash n$ be random partitions of~$n$.
Then $\chi^\la[\mu]\ne 0$ w.h.p.~as $n\to \infty$.
\end{conj}

Andrew Soffer's calculations show gradual decrease of the
probability $p(n)$ of zeroes in the character table of~$S_n$,
for $n>20$, from $p(20)\approx 0.394$ to $p(39)\approx 0.359$.
Furthermore, for large partitions the probability that
the character values are small seems to be rapidly decreasing.
For example, $q(20)\approx 0.06275$ and $q(37)\approx 0.020375$,
where $q(n)$ is the probability that the character is equal to~$1$.

\begin{op}\label{op:rand-char-1}
Let $\la, \mu \vdash n$ be random partitions of~$n$.
Find the asymptotic behavior of \,
$p_n  :=  P\bigl(\chi^\la[\mu]= 1\bigr)$.
\end{op}

The data we have suggests that the probability in the
open problem decreases mildly exponentially: $p_n < \exp[-n^\al]$,
for some~$\al>0$.
\medskip

\section{Proof of the Main Lemma}\label{s:main-proof}

Let us first restate the lemma using our notation.

\smallskip

\noindent
{\bf Lemma}~\ref{t:main} (Main Lemma)  \emph{
Let $\la, \mu \vdash n$, such that $\mu=\mu'$ and $\chi^\la[\wh \mu]\ne 0$.
Then \ts $\la \in \Ups(\mu)$.
}

\medskip

\begin{proof}
Recall Lemma~\ref{l:large-durfee}, and let
$$
\epsi \mu \, = \,
\chi^\mu[\wh \mu] \, = \, (-1)^{(n-\ds(\mu))/2}\ts.
$$
Recall also that the $S_n$ conjugacy class of cycle type~$\zeta$, when
$\zeta$ is a partition into distinct odd parts, splits into two
conjugacy classes in the alternating group~$A_n$, which
we denote by $\zeta^1$ and~$\zeta^2$.
There are two kinds of irreducible characters of $A_n$.
For each partition $\nu$ of $n$ such that $\nu=\nu'$ there are
two irreducible characters associated to $\nu$, which we denote
by $\al^{\nu+}$ and $\al^{\nu-}$; and for each  partition $\nu$ of $n$
such that $\nu \neq \nu'$ there is an irreducible character associated
to the pair $\nu$, $\nu'$, which we denote by $\al^\nu$.
These characters are related to irreducible characters of $S_n$ as indicated below.
We will need the following standard results (see e.g.~\cite[Section~2.5]{JK}).

\smallskip
\noindent
{\bf 1.} \. If $\nu \ne \nu'$, then
$$
\res{S_n}{A_n}{\chi^\nu} \, = \, \res{S_n}{A_n}{\irredsym {\nu'}}\, = \,\al^\nu \ts,
$$
is an irreducible character of~$A_n$.

\smallskip

\noindent
{\bf 2.} \. If $\nu=\nu'$, then
$$\res{S_n}{A_n}{\chi^\nu} \, = \, \al^{\nu+} \. + \. \al^{\nu-}\ts,
$$
is the sum of two different irreducible characters of $A_n$.  Moreover,
both characters are conjugate, that is, for any $\sigma\in A_n$ we have
$$
\al^{\nu+}\bigl[ (1\ts 2)\sigma(1\ts 2)\bigr]\, = \, \al^{\nu -}[\sigma].
$$

\smallskip

\noindent
{\bf 3.} \. The characters $\al^\nu$, $\nu \neq\nu'$ and $\al^{\nu+}$, $\al^{\nu-}$,
where $\nu=\nu'$ are all different and form a complete set of irreducible characters
of $A_n$.

\smallskip

\noindent
{\bf 4.} \. If $\nu=\nu'$, and $\gamma$ is a conjugacy class of $A_n$ different from $\wh\nu^1$ or $\wh\nu^2$, then
$$
\al^{\nu+}[\gamma] \, = \, \al^{\nu-}[\gamma] \, = \, \frac{1}{2} \ \chi^\nu[\gamma]\ts.
$$
We also have
$$
\aligned
\al^{\nu +}\bigl[\wh\nu^1\bigr] \, &= \, \al^{\nu-}\bigl[\wh\nu^2\bigr] \, = \, \frac{1}{2}\left( \ve_\nu +  \sqrt{\ve_\nu \prod_i \wh\nu_i} \right), \\
\al^{\nu +}\bigl[\wh\nu^2\bigr] \, &= \, \al^{\nu-}\bigl[\wh\nu^1\bigr] \, = \, \frac{1}{2}\left( \ve_\nu -  \sqrt{\ve_\nu \prod_i \wh\nu_i} \right).
\endaligned
$$
In other words, for any self-conjugate partition $\nu$, the only irreducible characters of $A_n$ that differ on
the classes $\wh\nu^1$ and $\wh\nu^2$ are precisely $\al^{\nu\pm}$.

\smallskip

There are two cases to consider with respect to whether $\la$ is self-conjugate or not.

First, assume that $\la\neq \la'$. Then $\al^\la$ is an irreducible character of~$A_n$.
Since
$$
\al^{\la} [\wh\mu^1] \.  = \. \al^\la[\wh\mu^2] \. = \. \chi^\la [\wh\mu]\ts,
$$
we obtain:
$$
\bigl(\al^{\mu +} \otimes \al^{\la}\bigr) [\wh\mu^1]
- \bigl(\al^{\mu +} \otimes \al^{\la}\bigr) [\wh\mu^2] \.
= \. \bigl(\al^{\mu+} [\wh\mu^1] - \al^{\mu+}[\wh\mu^2]\bigr) \cdot \chi^\la [\wh\mu] $$
$$= \left(\sqrt{\ve_\mu \. \prod_i \wh\mu_i}\right) \cdot \ts \chi^\la(\wh\mu) \. \neq \. 0\ts.
$$
Therefore, either $\al^{\mu+}$ or $\al^{\mu-}$ is a component of
$\al^{\mu+} \otimes \al^{\la}$. In other words,
$$\text{either} \quad \langle \irredalt{\mu+} \otimes \irredalt{\,\la},
\irredalt{\mu+} \rangle \neq 0 \quad \text{ or }\quad \langle \irredalt{\mu+} \otimes \irredalt{\,\la},
\irredalt{\mu-} \rangle \neq 0.$$
We claim that the terms in these product can be interchanged.  Formally, we claim that:
$$(\star) \qquad \text{either} \quad
\langle \al^{\mu+} \otimes \al^{\mu+}, \al^{\la} \rangle \neq 0 \quad
\text{or} \quad \langle \al^{\mu+} \otimes \al^{\mu-}, \al^{\la} \rangle \neq 0.$$
There are two cases.
If $\ve_\mu =1$, then both $\al^{\mu+}$ and $\al^{\mu-}$ take real values. Thus
$$
\langle \al^{\mu+} \otimes \al^{\mu\pm}, \al^{\la} \rangle \.
= \. \langle \al^{\mu+} \otimes \al^{\la}, \irredalt{\mu\pm} \rangle  \. \neq \ts 0\ts,
$$
which implies $(\star)$ in this case.

If $\ve_\mu =-1$, then
$$\im\left( \al^{\mu+}[\wh\mu^{1}]\right) \. = \. -\ts\im\left( \al^{\mu-}[\wh\mu^1]\right)
\quad \text{and} \quad\im\left( \al^{\mu+}[\wh\mu^{2}]\right) \. = \. -\ts\im\left( \al^{\mu-}[\wh\mu^2]\right).
$$
Therefore, $\overline{\al^{\mu+} }= \al^{\mu-}$, since all other character values are real.
Thus,
$$
\langle \al^{\mu+} \otimes \al^{\mu\pm}, \al^{\la} \rangle \.
= \. \langle \al^{\mu+} \otimes \al^{\la}, \al^{\mu\mp} \rangle \. \neq \ts 0,
$$
which implies $(\star)$ in this case.

In summary, we have both cases in~$(\star)$ imply that $\al^{\la}$ is a component of
$\res{S_n}{A_n}{\chi^\mu \otimes \chi^\mu}$.  Therefore, either $\chi^\la$ or
$\chi^{\la'}$ is a component of $\chi^\mu \otimes \chi^\mu$.
Since $\mu = \mu'$, we have, by Proposition~\ref{p:sign},
that $\chi^\la$ and $\chi^{\la'}$ are components of $\chi^\mu \otimes \chi^\mu$,
as desired.  This completes the proof of the $\la\ne \la'$ case.

\smallskip

Now, suppose $\la=\la'$.  The case $\la=\mu$ is given by Theorem~\ref{t:BB}.
If $\la\neq \mu$, then
$$\al^{\la\pm}[\wh\mu^1] \. = \. \al^{\la\pm}[\wh\mu^2] \. = \. \frac{1}{2} \chi^\la (\wh\mu) \neq 0\ts.$$
By a similar argument as above applied to $\ts\la+\ts$ and $\ts\la-\ts$ in place of~$\ts\la$,
we have the following analogue of~$(\star)$~:
$$\aligned
\text{either} \quad \langle \al^{\mu+} \otimes \al^{\mu+}, \al^{\la+} \rangle \.
&= \. \langle \al^{\mu+} \otimes \al^{\mu+}, \al^{\la-} \rangle  \. \neq \. 0\ts, \\
\text{or} \quad \langle \al^{\mu+} \otimes \al^{\mu-}, \al^{\la+} \rangle \.
&= \. \langle \al^{\mu+} \otimes \al^{\mu-}, \al^{\la-} \rangle  \. \neq \. 0\ts.
\endaligned
$$
This implies that  $\ts\al^{\la+}$\ts and  $\ts\al^{\la-}\ts$ are components of
$\ts\res{S_n}{A_n}{\chi^\mu \otimes \chi^\mu}$.  Therefore, $\chi^\la$ is
a component of $\chi^\mu \otimes \chi^\mu$, as desired.  This completes
the proof of the $\la= \la'$ case, and finishes the proof of the lemma.
\end{proof}

\medskip

\noindent

\section{Conclusions and final remarks}\label{s:fin}

\subsection{} \label{ss:fin-exp} \. For the staircase shapes~$\rho_k$,
the number $|\Ups(\rho_k)|$ of irreducible constituents is
exponential by Proposition~\ref{cor:ds-exp}.  From this point of view,
theorems~\ref{t:summary},~\ref{t:near-hook} and~\ref{t:near-two}
barely make a dent: they add $O(k^2)$ additional constituents.  On the
other hand, for the chopped square shape~$\eta_k$ there is no obvious
weakly exponential lower bound.  Although we believe that finding such bound
should not be difficult by an ad hoc construction, $\Omega(k^2)$
is the best bound we currently have.


\subsection{} \label{ss:fin-hooks} \. In~\cite{PaPa}, we obtain an
advance extension of Theorem~\ref{t:Bla}, based again on a combinatorial
interpretation given in~\cite{Bla}.  Among other things, we prove that
$\Phi(\rho_k)$ contain all hooks for \emph{all}~$k$, not just~$k$
large enough.   We should mention that this was independently
proved by Blasiak.\footnote{Personal communication.}


\subsection{} \. There is a curious characterization of positivity
Littlewood--Richardson coefficients of the staircase shape.
Namely, Berenstein and Zelevinsky proved in~\cite{BZ2}
the former \emph{Kostant Conjecture}, which states that
$LR(\rho_k,\rho_k,\la)>0$ if and only if $K(2\rho_k,\la)>0$.
For the proof, they defined \emph{BZ-triangles}, which proved
crucial in~\cite{KT}.


\subsection{} \label{ss:fin-dim} \. A natural question would be to ask
whether the dimensions of $\rho_k$, $\mon_k$ and~$\eta_k$ are large
enough to contain all irreducible representations of~$S_n$.
Same question for the limit shape defined in $\S$\ref{ss:part-limit} discussed
also in Open Problem~\ref{op:rand-limit}.  The answer is yes in all
cases, as can be seen by a direct application of the hook-length and
Stirling formulas.  Given that all these shapes are far from the
Kerov--Vershik shape which has the maximal (and most likely)
dimension~\cite{VK}, this might seem puzzling.
The explanation is that the dimensions are asymptotically greater than
the number of partitions $\pi(n) = \exp \Theta(\sqrt{n})$.

More precisely, it is well known and easy to see that the sum of all dimensions
is equal to $a_n = \#\{\si\in S_n \mid \si^2=1\}$. The sequence $\{a_n\}$ is
A000085 in~\cite{OEIS}, and is equal to $\exp \bigl[\frac{1}{2}n\log n]$,
up to an $\exp O(n)$ factor (see~\cite{Rob}).
Thus, the squares of dimensions of \emph{random} irreducible
representations of~$S_n$ are \emph{much larger} than~$a_n$.
This also suggests a positive answer to Open Problem~\ref{op:rand-limit}.
More relevant to the Tensor Square Conjecture, the dimensions of
$\rho_k$, $\mon_k$ and~$\eta_k$ are also equal to $\exp \bigl[\frac{1}{2}n\log n]$,
up to an $\exp O(n)$ factor, which supports the Saxl Conjecture and
conjectures~\ref{conj:chop},~\ref{conj:ups}.


\subsection{} \label{ss:fin-unim} \.
Although there are very strong results on the monotonicity of the
partition function (see e.g.~\cite{BE}), for partitions into distinct
parts much less is known (see~\cite{OR,RS}).
Curiously, some results follow from Dynkin's result in Lie Theory.
For example, the unimodality of $\pd_R(n)$ for $R=\{1,2,\ldots,k\}$,
i.e.~the unimodality of the coefficients in
$$
\prod_{i=1}^k \ts (1+t^i)\ts,
$$
is called
\emph{Hughes theorem} and corresponds to root system $C_k$.  We refer to~\cite{Bre,Sta-unim} for more
on this approach and general surveys on unimodality.


\subsection{} \label{ss:fin-asympt} \. Constants~$L$ in theorems~\ref{t:near-hook}
and~\ref{t:near-two} should be possible to estimate explicitly, in the same manner
as was done in~\cite{OR}, to reprove the Hughes theorem (see above).
However, because of Example~\ref{ex:const}, there is
no unimodality for small values of~$\ell$, so without advances in the study of
Kronecker coefficients it is unlikely that the gap can be bridged by an explicit
computation.


\subsection{} \label{ss:fin-back} \.
The reason conjectures~\ref{conj:rand-sc-char} and~\ref{conj:rand-char}
are not in direct contradiction has to do with a difference between
limit shapes of random partitions~$\mu$ and random partitions into
distinct odd part~$\wh\mu$ (see~\cite{Ver}).  Using the asymptotics
for the number of the latter, Conjecture~\ref{conj:rand-sc-char}
implies that at least $\exp \Theta(n^{1/4})$ many columns of the
character table $M_n$, most entries are zero.  In a different direction,
the Main Lemma can be reversed to show that at least $\exp \Theta(n^{1/4})$
many rows have most entries zero, by taking partitions with Durfee
squares $O(n^{1/4})$. We omit the details.


\subsection{} \label{ss:fin-tiling} \.
Curiously, and quite coincidentally, the plane can be tiled with copies of
parallel translations and rotations of~$\mon_k$ shapes, and the same is
true also for $\rho_k$ and $\eta_k$ (see Figure~\ref{f:tiling}).

\begin{figure}[hbtp]
   \includegraphics[scale=0.33]{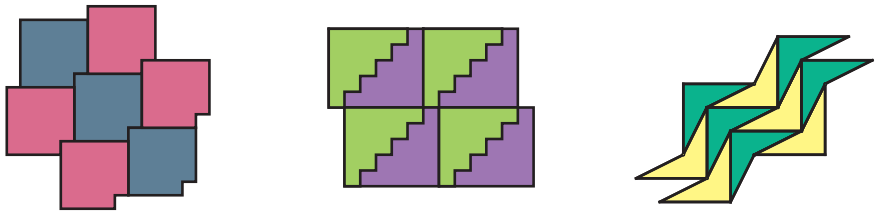}
   \caption{Tiling of the plane with chopped square, staircase and caret shapes.}
   \label{f:tiling}
\end{figure}


\subsection{} \label{ss:fin-back} \.
John Thompson's conjecture states that every finite simple group~$G$ has a
conjugacy class whose square is the whole~$G$.  For~$A_n$, this is a
well known result~\cite{Ber}.  We refer to~\cite{Sha} for a survey
of recent progress towards Thompson's conjecture, and related results.

Passman's problem
is concerned with the conjugation action on conjugacy classes of general
simple groups.  For~$A_n$, it was positively resolved by Heide and
Zalesski in~\cite{HZ}.

\begin{thm}[\cite{HZ}] For every $n\ge 5$ there is a conjugacy class $C_\la$
of~$A_n$, such that the action of~$A_n$ on $C_\la$ by conjugation as a
permutation module, contains every irreducible character as a constituent.
\end{thm}



\subsection{} \label{ss:fin-kron}
\. It is known that the problem \textsc{Kron} of computing $g(\la,\mu,\nu)$ is
\textsc{\#P}-hard (see~\cite{BI}).  However, the same also holds for the
Kostka numbers $K_{\la,\mu}$ and the Littlewood--Richardson coefficients
$LR(\la,\mu,\nu)$, so this is not the main obstacle~\cite{Nar}.
The difference is, for the latter there exists several combinatorial
interpretations in the form of counting certain Young tableaux,
which are (relatively) easy to work with (see e.g.~\cite{KT,KTW,PV,Sta}).
Of course, it is not known whether \textsc{Kron} is in~\textsc{\#P},
as this would imply a combinatorial interpretation for the
Kronecker coefficients, but it was shown in~\cite{BI} that
\textsc{Kron}$\ts \in \ts$\textsc{GapP}.

Recall that by the Knutson--Tao theorem~\cite{KT}
(formerly the \emph{saturation conjecture}~\cite{Zel}),
the problem whether $LR(\la,\mu,\nu)=0$ is equivalent to the problem whether
the corresponding LR-polytope is nonempty, and thus can be
solved in polynomial time by linear programming~\cite{MNS}.
Similarly, by the Knutson--Tao--Woodward theorem~\cite{KTW}
(formerly \emph{Fulton's conjecture}), the problem whether
$LR(\la,\mu,\nu)=1$ is equivalent to the problem whether
the corresponding LR-polytope consists of exactly one point,
and thus also can be solved in polynomial time (cf.~\cite{BZ}). Finally,
Narayanan showed that the corresponding problems for Kostka numbers reduce
to LR-coefficients~\cite{Nar}.  Together these results imply that all
four decision problems can be solved in polynomial time.

On the other hand, it is not known whether the corresponding decision
problems for Kronecker coefficients are in~\textsc{P}~\cite{Mul}.
In the words of Peter B\"urgisser~\cite{Bur},
``deciding positivity of Kronecker coefficients [...] is a
major obstacle for proceeding with geometric complexity theory''
of Mulmuley and Sohoni~\cite{MS}.  Special cases of the problem are
resolved in~\cite{PaPa}.  We refer to~\cite{Mul} for the detailed overview
of the role Kronecker coefficients play in this approach
(see also~\cite{BOR,Ike}).

Let us mention that when $n$ is in unary, computing individual character
values $\chi^\la[\mu]$ can be done in Probabilistic Polynomial time~\cite{Hel}.
This gives an $\exp O(\sqrt{n})$ time algorithm for computing Kronecker coefficients,
via the scalar product of characters.  On the other hand, the  Main Lemma now
can be viewed as a ``polynomial witness'' for positivity of Kronecker
coefficients, which, if Conjecture~\ref{conj:rand-rho} holds, works for
most partitions.

\vskip.65cm

{\small

\noindent
{\bf Acknowledgements.} \  We are very grateful to Jan Saxl for telling us
about his conjecture, and to Jonah Blasiak, Matt Boylan, Bob Guralnick,
Stephen DeSalvo, Christian Ikenmeyer, Hari Narayanan, Rosa Orellana,
Sinai Robins, Rapha\"{e}l Rouquier, Pham Huu Tiep and Gjergji Zaimi for
helpful remarks and useful references.  We would like to thank Andrew Soffer
for helpful computer computations.  We would like to credit to {\tt MathOverflow}
for providing a forum where such questions can be
asked.\footnote{Specifically, the references in the answers to this {\tt MathOverflow}
question proved very useful: {\tt http://mathoverflow.net/questions/111507/}}
The first author was partially supported by the BSF and the NSF grants,
the second by the Simons Postdoctoral Fellowship, and the third author
by UNAM-DGAPA grant IN102611-3.
}


\vskip1.1cm


\end{document}